\documentclass[12pt, reqno]{amsart}
\usepackage{amsmath,amssymb,amsfonts,amsthm}

 \usepackage[margin=1.4in]{geometry}

\usepackage{cleveref}
 \crefname{theorem}{Theorem}{Theorems}
 \crefname{thm}{Theorem}{Theorems}
 \crefname{lemma}{Lemma}{Lemmas}
 \crefname{lem}{Lemma}{Lemmas}
 \crefname{remark}{Remark}{Remarks}
 \crefname{proposition}{Proposition}{Propositions}
 \crefname{defn}{Definition}{Definitions}
 \crefname{corollary}{Corollary}{Corollaries}
 \crefname{section}{Section}{Sections}
 \crefname{figure}{Figure}{Figures}
 \newcommand{\crefrangeconjunction}{--}

\newtheorem{thm}{Theorem}
\newtheorem{theorem}[thm]{Theorem}

\newtheorem{corollary}[thm]{Corollary}

\theoremstyle{remark}
\newtheorem{remark}[thm]{Remark}

\numberwithin{equation}{section}

\newcommand{\Z}{\mathbb{Z}}
\newcommand{\E}{\mathbb{E}}
\newcommand{\eqd}{\overset{d}{=}}
\title[Random Walk Collisions in Reversible Random Graphs]{Collisions of Random Walks\\ in Reversible Random Graphs}
\author{Tom Hutchcroft}
\address{University of British Columbia}
\email{thutch@math.ubc.ca}
\author{Yuval Peres}
\address{Microsoft Research}
\email{peres@microsoft.com}

\begin{document}
\begin{abstract}
We  prove that in any recurrent reversible random rooted graph, two independent simple random walks started at the same vertex collide infinitely often almost surely. This applies to the Uniform Infinite Planar Triangulation and Quadrangulation
and to the Incipient Infinite Cluster in $\Z^2$.
\end{abstract}

\maketitle
\section{Introduction}
Let $G$ be an infinite, connected, locally finite graph. $G$ is said to have the \textbf{infinite collision property} if for every vertex $v$ of $G$, two independent random walks $\langle X_n \rangle_{n\geq0}$ and $\langle Y_n \rangle_{n\geq 0}$ started from $v$ collide (i.e. occupy the same vertex at the same time) infinitely often almost surely (a.s.). 
Although transitive recurrent graphs such as $\mathbb{Z}$ and $\mathbb{Z}^2$ are easily seen to have the infinite collision property, Krishnapur and Peres  \cite{KrPe04} showed that the infinite collision property does not hold for the \emph{comb graph}, a subgraph of $\Z^2$.



Chen and Chen \cite{ChenChen10} proved that the infinite cluster of supercritical Bernoulli bond percolation in $\Z^2$ a.s.~ has the infinite collision property. Barlow, Peres and Sousi \cite{BaPeSo12} gave a sufficient condition for the infinite collision property in terms of the Green function. They deduced that several classical random recurrent graphs have the infinite collision property, including the incipient infinite percolation cluster in dimensions $d\geq 19$.
However, 
 the methods of \cite{ChenChen10} and \cite{BaPeSo12} both require precise estimates on the graphs under consideration,  
 and the infinite collision property was still not known to hold for several important random recurrent graphs -- see \cref{C:examples}.


In this note we prove that the infinite collision property holds a.s.~ for a large class of random recurrent graphs.
Recall that a \textbf{rooted graph} $(G,\rho)$ is a graph $G$ together with a distinguished root vertex $\rho$, and that a random rooted graph $(G,\rho)$ is said to be \textbf{reversible} if $(G,\rho,X_1)$ and $(G,X_1,\rho)$ have the same distribution, where $X_1$ is the first step of a simple random walk on $G$ started at $\rho$ (see \cref{S:Definitions} for more details).
\begin{theorem}\label{T:Main} Let $(G,\rho)$ be a recurrent reversible random rooted graph. Then $G$ has the infinite collision property almost surely.
\end{theorem}
The assumption of reversibility can be replaced either by the assumption that $(G,\rho)$ is \emph{stationary} or by the assumption that $(G,\rho)$ is \emph{unimodular} and the root $\rho$ has finite expected degree (see \cref{S:Definitions} for definitions of these terms).


\begin{corollary}\label{C:examples} Each of the following graphs has the infinite collision property almost surely. (This is by no means an exhaustive list.)
\vspace{0.5em}
\begin{enumerate}\itemsep.75em
\item The Uniform Infinite Planar Triangulation (UIPT) and Quadrangulation (UIPQ).
\item The Incipient Infinite Cluster (IIC) of Bernoulli bond percolation in $\Z^2$.
\item Every component of each of the Wired Uniform and Minimal Spanning Forests (WUSF and WMSF) of any Cayley graph.
\end{enumerate}
\end{corollary}
The UIPT was introduced by Angel and Schramm \cite{UIPT1} and  the UIPQ was introduced by Krikun~ \cite{krikun2005local}. They are unimodular by construction and were shown to be recurrent by Nachmias and Gurel-Gurevich \cite{GN12}. The IIC is an infinite random subgraph of $\Z^2$ introduced by Kesten \cite{Kesten86}.  For each $n$, let $C_n$ be the largest cluster of a critical Bernoulli bond percolation on the box $[0,n]^2$ and let $\rho_n$ be a uniformly random vertex of $C_n$. J\'arai \cite{Jar03} showed that the IIC can be defined as the weak limit of the random rooted graphs $(C_n,\rho_n)$, and is therefore unimodular ~\cite[\S2]{AL07}. For background on the Wired Uniform and Minimal Spanning Forests, see Chapters 10 and 11 of \cite{LP:book} and Section 7 of \cite{AL07}.

 In \cref{S:Extensions} we provide extensions of \cref{T:Main} to networks and to the continuous-time random walk.

\begin{remark}  If $G$ is a non-bipartite graph with the infinite collision property, it is easy to see that two independent random walks started from \emph{any} two vertices of $G$ will collide infinitely often a.s. On the other hand, if $G$ is a bipartite graph with the infinite collision property, then two independent random walks on $G$ will collide infinitely often if and only if their starting points are at an even distance from each other.\end{remark}

\subsection{Definitions}\label{S:Definitions}
In this section we give concise definitions of stationary, reversible and unimodular random rooted graphs. We refer the reader to Aldous and Lyons \cite{AL07} for more details.

A \textbf{rooted graph} $(G,\rho)$ is a connected, locally finite (multi)graph $G=(V,E)$ together with a distinguished vertex $\rho$, the \textbf{root}. An isomorphism of graphs $\phi: G \rightarrow G'$ is an isomorphism of rooted graphs $\phi:(G,\rho)\to(G',\rho')$ if $\phi(\rho)=\rho'$. The set of isomorphism classes of rooted graphs is endowed with the \textbf{local topology} ~\cite{BeSc}, in which, roughly speaking, two (isomorphism classes of) rooted graphs are close to each other if and only if they have large isomorphic balls around the root. A \textbf{random rooted graph} is a random variable taking values in the space of isomorphism classes of rooted graphs endowed with the local topology. 
Similarly, a \textbf{doubly-rooted graph} is a graph together with an ordered pair of distinguished (not necessarily distinct) vertices. Denote the space of isomorphism classes of doubly-rooted graphs equipped with this topology by $\mathcal{G}_{\bullet\bullet}$.

Recall that the \textbf{simple random walk} on a locally finite (multi)graph $G=(V,E)$ is the Markov process $\langle X_n\rangle_{n\geq0}$ on the state space $V$ with transition probabilities $p(u,v)$ defined to be the fraction of edges emanating from $u$ that end in $v$. 
A random rooted graph $(G,\rho)$ is said to be \textbf{stationary} if, when $\langle X_n \rangle_{n\geq 0}$ is a simple random walk on $G$ started at the root, \[(G,\rho) \eqd (G,X_n)\]
for all $n$ and is said to be \textbf{reversible} if 
\[ (G,\rho,X_n) \eqd (G,X_n,\rho)\]
for all $n$. Every reversible random rooted graph is clearly stationary, but the converse need not hold in general \cite[Examples 3.1 and 3.2]{BLPS99}. However, Benjamini and Curien \cite[Theorem 4.3]{BC2011} showed that every recurrent stationary random rooted graph is necessarily reversible, so that \cref{T:Main} also applies under the apparently weaker assumption of stationarity.

Reversibility is closely related to the property of unimodularity.  A \textbf{mass transport} is a function $f:\mathcal{G}_{\bullet\bullet}\to[0,\infty]$. A random rooted graph $(G,\rho)$ is said to be \textbf{unimodular} if it satisfies the \textbf{Mass-Transport Principle}: for every mass transport $f$,
\begin{equation*}\label{eq:MTP}\tag{MTP} \E\left[\sum_v f(G,\rho,v)\right] = \E\left[\sum_uf(G,u,\rho)\right]. \end{equation*} 
That is, 
\[ \textit{Expected mass out equals expected mass in.} \]
The Mass-Transport Principle was first introduced by H\"aggstr\"om \cite{Hagg97} to study dependent percolation on Cayley graphs. The current formulation of the Mass-Transport Principle was suggested by Benjamini and Schramm \cite{BeSc} and developed systematically in \cite{AL07}.


As noted in \cite{BC2011}, if $(G,\rho)$ is a unimodular random rooted graph with $\E[\deg(\rho)]<\infty$, then biasing the law of $(G,\rho)$ by $\deg(\rho)$ 
gives an equivalent law of a reversible random rooted graph. Conversely, if $(G,\rho)$ is a reversible random rooted graph, then biasing the law of $(G,\rho)$ by $\deg(\rho)^{-1}$ gives an equivalent law of a reversible random rooted graph. 
For example, if $(G,\rho)$ is a finite random rooted graph then it is unimodular if and only if $\rho$ is uniformly distributed on $G$, and is reversible if and only if $\rho$ is distributed according to the stationary measure of simple random walk on $G$. 

In light of the above correspondence, \cref{T:Main} may be stated equivalently as follows.
\begin{theorem}\label{T:Mainunimod} Let $(G,\rho)$ be a recurrent unimodular random rooted graph with $\E[\deg(\rho)]<\infty$. Then $G$ has the infinite collision property almost surely.
\end{theorem}

 We provide two variations on the proof of \cref{T:Main}. The first uses the Mass-Transport Principle. The second, given in \cref{S:Extensions}, uses reversibility, and applies in the network setting also.

\section{Proof of \cref{T:Main}}
\begin{proof}

Let $G$ be a graph and let $p_n(\hspace{.1em}\cdot\hspace{.225em},\hspace{.05em}\cdot\hspace{.15em})$ denote the $n$-step transition probabilities for simple random walk on $G$. For each vertex $u$ of $G$, let $q_\text{fin}(u)$ denote the probability that two independent random walks started at $u$ collide only finitely often, and let $q_0(u)$ denote the probability that two independent random walks started at $u$ do not collide at all after time zero. Finally, for each pair of vertices $u$ and $v$ let $q_\text{last}(u,v)$ be the probability that two independent random walks started at $u$ collide for the last time at $v$, so that
\[ q_\text{fin}(u)=\sum_vq_\text{last}(u,v).\]
Decomposing according to the time of the last collision gives
\begin{equation}\label{eq:lastcollision1} q_\text{last}(u,v) = \sum_{n\geq0} p_n(u,v)^2q_0(v)\end{equation}
and hence
\begin{equation}\label{eq:lastcollision} q_\text{fin}(u) = \sum_v\sum_{n\geq0} p_n(u,v)^2q_0(v).\end{equation}


%

Suppose that $(G,\rho)$ is a recurrent unimodular random rooted graph with $\E[\deg(\rho)]<\infty$.
Consider the mass transport \[f(G,u,v) = \deg(u)q_\text{last}(u,v).\] 
Each vertex $u$ sends a total mass of $\deg(u)q_\text{fin}(u)$, while, by \eqref{eq:lastcollision1}, each vertex $v$ receives a total mass of
\begin{align*} \sum_uf(G,u,v) &= \sum_{n\geq0}\sum_{u}\deg(u)p_n(u,v)^{2}q_0(v)\\
&= q_0(v)\sum_{n\geq0}\sum_u\deg(v)p_n(v,u)p_n(u,v)\\
&= q_0(v)\deg(v)\sum_{n\geq0}p_{2n}(v,v).\end{align*}
Since $G$ is recurrent, the sum $\sum_{n\geq0}p_{2n}(v,v)$ is infinite a.s.~ for every vertex $v$ of $G$. By the Mass-Transport Principle,
\[\E[\deg(\rho)q_\text{fin}(\rho)] = \E\left[q_0(\rho)\deg(\rho)\sum_{n\geq0}p_{2n}(\rho,\rho)\right]. \]
Since the left-hand expectation is finite by assumption, we must have that $q_0(\rho)=0$ a.s., and consequently that $q_\text{fin}(v)=0$ for every vertex $v$ in $G$ a.s.
\end{proof}

\section{Extensions}\label{S:Extensions}
\subsection{Networks}
Recall that a \textbf{network} $(G,c)$ is a connected locally finite graph $G=(V,E)$ together with a function $c:E\to(0,\infty)$ assigning to each edge $e$ of $G$ a positive \textbf{conductance} $c(e)$. Graphs may be considered to be networks by setting $c(e) \equiv1$. Write $c(u)$ for the sum of the conductances of the edges emanating from $u$ and $c(u,v)$ for the sum of the conductances of the edges joining $u$ and $v$. The random walk $\langle X_n \rangle_{n\geq 0}$ on a network $(G,c)$ is the Markov chain on $V$ with transition probabilities $p(u,v) = c(u,v)/c(u)$.  Unimodular and reversible  random rooted networks are defined similarly to the unweighted case ~\cite{AL07}. In particular, a random rooted network $(G,c,\rho)$ is defined to be reversible if, letting $\langle X_n \rangle_{n\geq 0}$ be a random walk on $(G,c)$ started from $\rho$,
\[(G,c,\rho,X_n) \eqd (G,c,X_n,\rho) \text{ for all $n\geq1$.}\]


We now extend \cref{T:Main} to the setting of reversible random rooted networks. The proof given also yields an alternative proof of \cref{T:Main}.
\begin{theorem}\label{T:Network} Let $(G,c,\rho)$ be a recurrent reversible random rooted network. Then $(G,c)$ has the infinite collision property almost surely.
\end{theorem}

\begin{proof}
Let $q_\text{fin}$ and $q_0$ be defined as in the proof of \cref{T:Main}.
Taking expectations on both sides of \eqref{eq:lastcollision} with $u=\rho$,
\[ \E\left[q_\text{fin}(\rho)\right]=\E\left[\sum_v\sum_{n\geq 0}p_n(\rho,v)^2q_0(v)\right]=\sum_{n\geq 0}\E\big[p_n(\rho,X_n)q_0(X_n)\big].\]
Applying reversibility, 
\[ \E[q_\text{fin}(\rho)]=\sum_{n\geq 0}\E\big[p_n(X_n,\rho)q_0(\rho)\big] = \E\left[q_0(\rho)\sum_{n\geq0}p_{2n}(\rho,\rho)\right]. \]
Since $G$ is recurrent, $\sum_{n\geq0}p_{2n}(v,v)=\infty$ a.s.~ for every vertex $v$ of $G$. Thus, since the left-hand expectation is finite, we must have that $q_0(\rho)=0$ a.s.~ and hence also that $q_\text{fin}(\rho)=0$ a.s.
\end{proof}


\subsection{Continuous time random walk}
Let $(G,c,\rho)$ be a recurrent unimodular random rooted network, and let $\langle X_t \rangle_{t\geq0}$ denote the continuous-time random walk started from $\rho$, which jumps across each edge $e$ with rate $c(e)$ (see e.g.~ \cite{NorrisBook}). By \cite[Corollary 4.3]{AL07}, $(G,c,\rho)$ is reversible for $\langle X_t \rangle_{t\geq0}$ in the sense that \[(G,\rho,X_t)\eqd(G,X_t,\rho)\] for all $t\geq 0$. (The continuous-time walk on a recurrent graph a.s.~ does not explode). 

The natural analogue of the infinite collision property also holds a.s.~ for the continuous-time random walk on recurrent unimodular random rooted networks.
\begin{theorem}\label{T:continuoustime} let $(G,c,\rho)$ be a recurrent unimodular random rooted network and let $\langle X_t\rangle_{t\geq 0}$ and $\langle Y_t \rangle_{t \geq 0}$ be two independent continuous-time random walks on $(G,c)$ started from any two vertices. Then the set of times $\{t: X_t = Y_t\}$ has infinite Lebesgue measure almost surely. 
\end{theorem}

\begin{proof} Let $\langle X_t \rangle_{t\geq0}$ and $\langle Y_t\rangle_{t\geq0}$ be independent continuous-time random walks starting at the same vertex $u$ of $G$. By considering the walks only at integer times $t=n$, the proof of \cref{T:Network} readily shows that the set of integer collision times $\{n \in \mathbb{N}: X_n=Y_n\}$ is infinite a.s. 


For every $s\geq 0$ there is a positive probability that neither $\langle X_t\rangle_{t\geq0}$ nor $\langle Y_t \rangle_{t\geq0}$ has made any jumps by time $s$. Thus, the law of the sequence $\langle (X_{n+s},Y_{n+s}) \rangle_{n \geq 1}$ is absolutely continuous with respect to the law of $\langle (X_{n}, Y_{n}) \rangle_{n \geq 1}$.
It follows that for every $s\geq0$, the walks $\langle X_t \rangle_{t\geq0}$ and $\langle Y_t\rangle_{t\geq0}$ collide at an infinite set of times of the form $\{n+s : n \in \mathbb{N}\}$ a.s., and consequently that 
\[ \mathcal{L}\hspace{-.055em}\textit{eb}\big(\{t: X_t = Y_t\}\big)=\int_0^1\!\big|\{n : X_{n+s}=Y_{n+s}\}\big|\, \mathrm{d}s =\infty \quad \text{a.s.}\]

For every other vertex $v$, there is a positive probability that $X_1=u$ and $Y_1=v$, so that the law of two independent continuous-time random walks started from $u$ and $v$ is absolutely continuous with respect to the law of $\langle (X_{t+1}, Y_{t+1}) \rangle_{t\geq0}$. Thus, the set of collision times of two independent continuous-time random walks started from $u$ and $v$ has infinite Lebesgue measure a.s.
\end{proof}

\begin{remark}
\cref{T:continuoustime} has consequences for the \emph{voter model} on unimodular random recurrent networks. For every network $(G,c)$ in which two independent continuous-time random walks collide a.s., duality between the voter model and continuous-time coalescing random walk (\cite[\S5]{Ligg05} and \cite[\S14]{aldous-fill-2014}) implies that   the only ergodic stationary measures for the voter model on $(G,c)$ are the constant (a.k.a.~ consensus) measures. Thus, a consequence of \cref{T:continuoustime} is that this holds for the voter model on recurrent unimodular random rooted networks. 
\end{remark}
\subsection*{Acknowledgements}
This work was carried out while TH was an intern at Microsoft Research. We thank Itai Benjamini for suggesting this problem, and also thank Lewis Bowen, Perla Sousi and Omer Tamuz for helpful discussions.
\bibliographystyle{abbrv}
 \bibliography{Collisions}
\end{document}